\documentclass{article}
\usepackage{amsmath,amsthm,amsfonts,amssymb,ytableau,mathrsfs}
\usepackage{authblk}
\usepackage[shortlabels]{enumitem}

\newtheorem{definition}{Definition}
\newtheorem{lemma}{Lemma}
\newtheorem{proposition}{Proposition}
\newtheorem{theorem}{Theorem}
\newcommand{\LL}{{\mathscr{L}}}
\newcommand{\lij}[3]{#1|\!\overline{\,\vphantom{i}#2,#3}}
\DeclareMathOperator{\shape}{shape}

\title{The Latin Tableau Conjecture}
\author[1]{Timothy Y. Chow}
\author[2]{Mark G. Tiefenbruck}
\affil[1]{Center for Communications Research---Princeton, NJ}
\affil[2]{Center for Communications Research---La Jolla, CA}
\begin{document}
\maketitle

\begin{abstract}
A \emph{Latin tableau of shape~$\lambda$ and type~$\mu$}
is a Young diagram of shape~$\lambda$ in which each box contains a
single positive integer, with no repeated integers in any row or column,
and the $i$th most common integer appearing $\mu_i$ times.
Over twenty years ago, Chow et~al., in their study of a generalization
of Rota's basis conjecture that they called the wide partition conjecture,
conjectured a necessary and sufficient condition for the existence of
a Latin tableau of shape~$\lambda$ and type~$\mu$.
We report some computational evidence for this conjecture,
and prove that the conjecture correctly characterizes,
for any given~$\lambda$, at least the first four parts of~$\mu$.
\end{abstract}

\section{Introduction}

Over twenty years ago, Chow et~al.~\cite{CFGV} proposed
the \emph{wide partition conjecture},
which generalized Gian-Carlo Rota's famous \emph{basis conjecture}~\cite{Rot}.
They hoped that enlarging the context of the basis conjecture
would suggest new lines of attack (e.g., proofs by induction
on the number of boxes in the Young diagram of a partition).
Unfortunately, even the special case of the wide partition conjecture
that Chow et~al.\ called the \emph{wide partition conjecture
for free matroids} has shown itself to be difficult,
and remains open to this day.

Chow et~al.\ mentioned in passing something
they called the \emph{Latin Tableau Question};
a positive answer to the question would imply
the wide partition conjecture for free matroids.
We believe that the Latin Tableau Question, in addition to having
interesting connections with Rota's basis conjecture,
is interesting in its own right.
In the present paper, we give a self-contained presentation of the 
Latin Tableau Question (that requires no knowledge of Rota's basis
conjecture or the wide partition conjecture),
and we prove some partial results.

We begin by defining a \emph{Latin tableau}.
A \emph{partition}~$\lambda$ is a weakly decreasing finite sequence
$\lambda_1 \ge \lambda_2 \ge \cdots \ge \lambda_\ell$ of positive integers.
The individual $\lambda_i$ are called the \emph{parts} of the partition,
and the sum of the parts is denoted $|\lambda|$.
The \emph{Young diagram} or simply \emph{diagram} of a partition~$\lambda$
is a left-justified grid of boxes whose $i$th row contains $\lambda_i$ boxes.
If we color each box of the diagram of~$\lambda$ in such a way
that there are no repeated colors in any row or column,
then we call the colored diagram~$T$ a \emph{Latin tableau of shape~$\lambda$}.
The \emph{type} of~$T$ is the partition~$\mu$ where
$\mu_1$ is the number of times that the most frequently occurring color appears,
$\mu_2$ is the number of times that the next most frequently occurring color appears,
and so on.  See Figure~\ref{fig:youngdiagram} for an example.

\begin{figure}[ht]
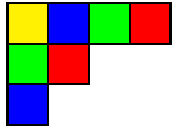

\begin{center}
\ytableausetup{centertableaux}
\begin{ytableau}
*(yellow) & *(blue) & *(green) & *(red) \\
*(green) & *(red) \\
*(blue)
\end{ytableau}
\end{center}
\caption{Latin tableau of shape $(4,2,1)$ and type $(2,2,2,1)$}
\label{fig:youngdiagram}
\end{figure}

The word \emph{Latin} is inspired by the well-known notion of
a \emph{Latin square}, but we caution the reader that even if $\lambda$
is a square shape, a Latin tableau of shape~$\lambda$ is not required
to have type$(\mu) = \lambda$, and hence is not necessarily a Latin square
in the usual sense. We can now state the following natural question.

\begin{quote}
\textbf{Given $\lambda$ and~$\mu$, what are necessary and sufficient
conditions for there to exist a Latin tableau of shape~$\lambda$
and type~$\mu$?}
\end{quote}

\noindent A conjectural answer to the above question was proposed
in \cite[Section 6]{CFGV}. To state this \emph{Latin Tableau
Conjecture}, we need a few more definitions.

Given a partition~$\lambda$ and a nonnegative integer~$n$, we let
$\alpha_n$ denote the maximum number of boxes of the diagram of~$\lambda$
that we can color using $n$ colors
without repeating a color in any row or column.
(Technically, we should write $\alpha_n(\lambda)$ to emphasize
the dependence of~$\alpha_n$ on~$\lambda$, but to avoid cluttering
the notation, we just write~$\alpha_n$ when there is no
danger of confusion.)  We define the \emph{chromatic difference sequence}
$\delta = (\delta_1, \delta_2, \ldots)$ by letting
\begin{equation*}
\delta_n := \alpha_n - \alpha_{n-1}
\end{equation*}
for each positive integer~$n$.  It is clear from the definitions that
if there exists a Latin tableau of shape~$\lambda$ and type~$\mu$,
then necessarily
\begin{align*}
\delta_1 = \alpha_1 &\ge \mu_1 \\
\delta_2 + \delta_1 = \alpha_2 &\ge \mu_1 + \mu_2 \\
\delta_3 + \delta_2 + \delta_1 = \alpha_3 &\ge \mu_1 + \mu_2 + \mu_3 \\
&\vdots
\end{align*}
i.e., $\delta$ \emph{dominates} or \emph{majorizes}~$\mu$
(written $\delta \succeq \mu$).  The Latin Tableau Conjecture
asserts that this necessary condition is sufficient.

\begin{quote}
\textbf{Latin Tableau Conjecture.} A Latin tableau of shape~$\lambda$
and type~$\mu$ exists if and only if $\delta\succeq \mu$.
\end{quote}

\noindent
In fact, Proposition~\ref{prop:dominance} below 
reduces the conjecture to showing that there
always exists a Latin tableau of shape~$\lambda$ and type~$\delta$.

More than once, when we have shown the Latin Tableau Conjecture
to someone, the reaction has been that such a natural statement
must be either known or false.
We are confident that it is not known;
Chow et~al.~\cite{CFGV} surveyed a lot of closely related literature,
and in the present paper, we establish further connections
with the literature on matching theory (see the proof of Theorem~\ref{thm:corner})
and chromatic difference sequences.
If anything, the prior work in these related areas
makes the Latin Tableau Conjecture seem unexpected,
since even slight generalizations are known to be false
(see the discussion of Figure~\ref{fig:skewshape} below).

The conjecture could of course be false,
but we have checked that it
is true for all partitions~$\lambda$ that fit inside a $12\times 12$ square.
Moreover, the main result of this paper is that
there always exists a Latin tableau of
shape~$\lambda$ and type~$\mu$ such that $\mu_i = \delta_i$ for
$1 \le i \le 4$.
There is some hope that our methods might generalize to prove the full
conjecture, but we currently rely on case analyses that become
increasingly complicated for higher terms of the CDS.

\section{Graph-Theoretic Perspective}
\label{sec:graphtheory}

In this section, we describe two different ways to think of a
Young diagram as a graph. Each way has its own merits.

All our graphs $G = \{V,E\}$ will be finite, simple, and undirected.
A \emph{coloring} of~$G$ is a map $\kappa \colon V \to \mathbb{N}$
such that $\kappa(u) \ne \kappa(v)$ whenever $u$ and~$v$ are adjacent.
(In the literature, the term \emph{proper coloring} is sometimes used,
but we will have no occasion to refer to any other type of coloring,
so we drop the word \emph{proper}.)
A \emph{stable set} or \emph{independent set} of~$G$ is a subset
$S\subseteq V$ such that whenever $u\in S$ and $v\in S$ then there
is no edge between $u$ and~$v$.
Observe that each \emph{color class} of a coloring~$\kappa$
(meaning the inverse image of some $n\in\mathbb{N}$)
is a stable set.
Given a coloring~$\kappa$, we define $\shape(\kappa)$ to be
the sequence of cardinalities of the color classes of~$\kappa$,
arranged in weakly decreasing order.

A very important concept for us is the chromatic difference
sequence of a graph~\cite{AB}.

\begin{definition}
Given a graph~$G$, we let $\alpha_i(G)$ (or simply $\alpha_i$,
if the graph under consideration is understood) denote the
maximum cardinality of a union of $i$ (disjoint) stable sets of~$G$.
The \emph{chromatic difference sequence} or \emph{CDS} of~$G$ is
the sequence $(\delta_1, \delta_2, \delta_3, \ldots)$ defined by
\begin{equation*}
\delta_i := \alpha_i - \alpha_{i-1}, \qquad \mbox{for $i\ge 1$.}
\end{equation*}
We say that $G$ is \emph{CDS-colorable} if there exists a coloring~$\kappa$
of~$G$ such that $\shape(\kappa)$ equals the CDS of~$G$.
\end{definition}

The first of our two ways of associating a graph with a partition
is given by the following definition, which establishes
a connection between the above definition of a CDS and the
definition given in the Introduction.

\begin{definition}
Given a partition~$\lambda$, the
\emph{partition graph}~$G_\lambda$ is the graph whose vertices
are the boxes of the diagram of~$\lambda$, and in which two boxes
are adjacent if and only if they lie in the same row or column.
\end{definition}

The CDS of an arbitrary graph is not necessarily weakly decreasing
(see for example \cite[Remark~2]{dW} or \cite[Figure~1]{AB}).
However, we have the following result.

\begin{proposition}
\label{prop:deWerra}
The CDS of a partition graph is itself a partition; i.e.,
it is weakly decreasing.
\end{proposition}

\begin{proof}
This was proved by de Werra \cite[Lemma 2.1]{dW},
who actually showed that the result holds for the
line graph of any bipartite graph.
\end{proof}

In particular, the CDS of~$G_\lambda$ coincides with the CDS of~$\lambda$
as defined in the Introduction, and the Latin Tableau Conjecture implies
(and in fact, by Proposition~\ref{prop:dominance} below, is equivalent to)
the CDS-colorability of~$G_\lambda$ for all partitions~$\lambda$.

It is well known that not every graph is CDS-colorable;
indeed, despite the considerable literature on chromatic
difference sequences, very few graphs are known to be CDS-colorable.
In Figure~\ref{fig:nondominable}, we regard the picture on the
left as a graph whose vertices are the boxes and in which two boxes
are adjacent if and only if they lie in the same row or column.
The other two pictures in Figure~\ref{fig:nondominable}
demonstrate that $\alpha_1 = 5$ and $\alpha_2 = 4 + 4 = 8$.
So the CDS is $(5,3,1)$,
but it is easy to check that there is no coloring~$\kappa$
such that $\shape(\kappa) = (5,3,1)$.
A slightly larger graph (Figure~\ref{fig:skewshape},
reproduced from~\cite{CFGV})
shows that that Latin Tableau Conjecture cannot be generalized
even to skew Young diagrams; the CDS is $(6,6,3,1)$
but there is no coloring~$\kappa$
such that $\shape(\kappa) = (6,6,3,1)$.

\begin{figure}[ht]
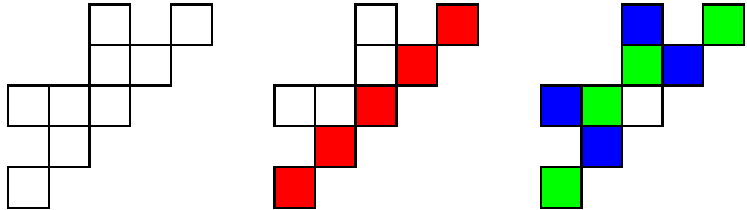

\begin{center}
\ytableausetup{centertableaux}
\begin{ytableau}
\none & \none & \    & \none & \ \\
\none & \none & \    & \         \\
\     & \     & \                \\
\none & \                        \\
\                                \\
\end{ytableau}
\qquad
\begin{ytableau}
\none & \none & \    & \none & *(red) \\
\none & \none & \    & *(red)         \\
\     & \     & *(red)                \\
\none & *(red)                        \\
*(red)                                \\
\end{ytableau}
\qquad
\begin{ytableau}
\none   & \none    & *(blue)    & \none   & *(green) \\
\none   & \none    & *(green)   & *(blue)            \\
*(blue) & *(green) & \                               \\
\none   & *(blue)                                    \\
*(green)                                             \\
\end{ytableau}
\end{center}
\caption{A non-CDS-colorable graph}
\label{fig:nondominable}
\end{figure}

\begin{figure}[ht]
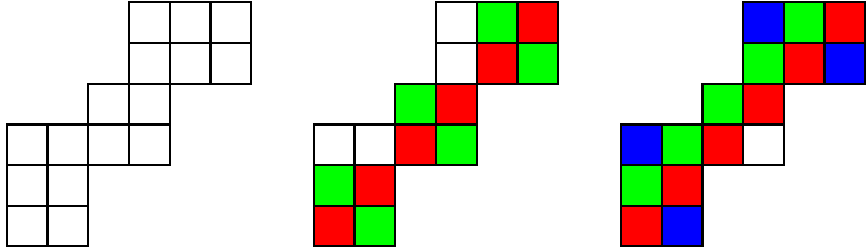

\begin{center}
\ytableausetup{centertableaux}
\begin{ytableau}
\none & \none & \none & \     & \     & \     \\
\none & \none & \none & \     & \     & \     \\
\none & \none & \     & \         \\
\     & \     & \     & \         \\
\     & \                        \\
\     & \
\end{ytableau}
\qquad
\begin{ytableau}
\none    & \none & \none    & \        & *(green) & *(red)    \\
\none    & \none & \none    & \        & *(red)   & *(green)  \\
\none    & \none & *(green) & *(red)          \\
\        & \     & *(red)   & *(green)        \\
*(green) & *(red)        \\
*(red)   & *(green)
\end{ytableau}
\qquad
\begin{ytableau}
\none    & \none    & \none    & *(blue)  & *(green) & *(red)    \\
\none    & \none    & \none    & *(green) & *(red)   & *(blue)  \\
\none    & \none    & *(green) & *(red)         \\
*(blue)  & *(green) & *(red)   & \              \\
*(green) & *(red)        \\
*(red)   & *(blue)
\end{ytableau}
\end{center}
\caption{A non-CDS-colorable skew shape}
\label{fig:skewshape}
\end{figure}

There is a second way to think of a Young diagram as a graph.

\begin{definition}
Given a partition~$\lambda$, we let
$B_\lambda$ be the bipartite graph whose vertices
are the rows and columns of the diagram of~$\lambda$, and in which
row~$i$ is adjacent to column~$j$ if and only if the diagram of~$\lambda$
contains a box in row~$i$ and column~$j$.
\end{definition}

The bipartite graph~$B_\lambda$ is a useful tool for proving
the following Proposition, which, together with Proposition~\ref{prop:deWerra},
implies that if we can prove the Latin Tableau Conjecture
in the case that $\mu = \delta$, then the full conjecture follows.

\begin{proposition}
\label{prop:dominance}
If there exists a Latin tableau with shape~$\lambda$ and type~$\mu$,
and $\mu\succeq \nu$, then
there exists a Latin tableau with shape~$\lambda$ and type~$\nu$.
\end{proposition}

\begin{proof}
We need the following fact about the dominance
order on partitions~\cite[Chapter~I, (1.16)]{Mac}.
If $\mu\succeq\nu$,
then we can transform $\mu$ to~$\nu$ by performing a
(carefully chosen) sequence of operations of the following type:
take two parts of the current partition that differ in size---call
them $a$ and~$b$, with $a<b$---and replace them with $a+1$ and $b-1$
(reordering the parts into weakly decreasing order if necessary)
to get a new partition.  Given this fact, it suffices to show that
if $T$ is a Latin tableau with shape~$\lambda$ and type~$\mu$,
and colors $i$ and~$j$ occur with multiplicities $\mu_i$ and~$\mu_j$
respectively, where $\mu_i < \mu_j$, then we can recolor those boxes
in such a way that color~$i$ occurs with multiplicity $\mu_i+1$ and
color~$j$ occurs with multiplicity $\mu_j-1$.

Now consider the bipartite graph~$B_\lambda$,
and identify the coloring of the boxes of~$T$
with an edge coloring of~$B_\lambda$.
In~$B_\lambda$, the edges
with color~$i$ form a \emph{matching} (i.e., a set of edges,
no two of which touch each other), and similarly the edges with
color~$j$ form a matching that is disjoint from the first.
The union of these two matchings form a subgraph of~$B_\lambda$ in which
every vertex has degree 1 or~2, and hence each connected component
must be a cycle or a path with edges of alternating colors.
Since $\mu_i < \mu_j$, at least one of these connected components
must be a path with an odd number of edges
and with color~$j$ edges in the majority.
Now we can simply reverse the colors of the edges in this path,
and observe that this decreases the number of color~$j$ edges
by~1, increases the number of color~$i$ edges by~1, and
preserves the Latin tableau property.
\end{proof}

% In light of Proposition~\ref{prop:dominance},
%^ for the remainder of the paper, we think of the Latin Tableau Conjecture
% as the assertion that $G_\lambda$ is CDS-colorable for any partition~$\lambda$.

\section{The CDS of a Partition Graph}
\label{sec:CDSpartition}

It will be convenient to coordinatize a Young diagram;
box $(i,j)$ denotes the box in row~$i$ and column~$j$,
where we number the rows and columns starting with~$0$
(i.e., the top left box is $(0,0)$).

\begin{definition}
For $k\ge 0$, \emph{antidiagonal~$k$}
is the set of boxes \hbox{$\{(0,k)$}, \hbox{$(1,k-1)$},
\hbox{$(2,k-2)$}, \dots, \hbox{$(k,0)\}$}
whose coordinates sum to~$k$.
\end{definition}

The portion of a Young diagram lying (weakly) below and to the right
of a box $(i,j)$ plays an important role, so we introduce the following
notation for it.

\begin{definition}
For a partition $\lambda$ and nonnegative integers $i$ and~$j$, we define
\begin{equation*}
\lij{\lambda}{i}{j} := \{ (i',j') \in \lambda \mid \mbox{$i'\ge i$ and $j'\ge j$} \},
\end{equation*}
where ``$(i',j') \in \lambda$'' means that $(i',j')$ are the coordinates
of a box in the diagram of~$\lambda$.
The notation $\lij{\lambda}{i}{j}$ is pronounced,
``$\lambda$ corner $i,j$.''
\end{definition}

The following Lemma states a simple but important
fact about $\lij{\lambda}{i}{j}$.

\begin{lemma}
\label{lem:lij}
A stable set in a partition graph~$G_\lambda$ contains at most $i+j$
boxes outside $\lij{\lambda}{i}{j}$.
\end{lemma}

\begin{proof}
The portion of the diagram of~$\lambda$ outside $\lij{\lambda}{i}{j}$
is contained in the union of $i$ rows and $j$ columns.
But at most $i$ elements of the stable set can lie in $i$ rows
and at most $j$ elements of the stable set can lie in $j$ columns.
\end{proof}

Lemma~\ref{lem:lij} implies a useful characterization of~$\delta_1$,
the first term of the CDS.

\begin{proposition}
\label{prop:cdsk}
Let $k$ be a positive integer.
Then the CDS of a partition graph~$G_\lambda$ begins with~$k$
if and only if
\begin{enumerate}
\item\label{item:cdsk1}
the diagram of~$\lambda$ contains the entirety of antidiagonal $k-1$, and
\item\label{item:cdsk2}
the diagram of~$\lambda$ omits some box from antidiagonal~$k$.
\end{enumerate}
\end{proposition}

\begin{proof}
Suppose that conditions \ref{item:cdsk1} and~\ref{item:cdsk2} hold.
Condition~\ref{item:cdsk1} implies that there exists a stable set
with $k$ boxes---simply take antidiagonal $k-1$.
Let $(i,j)$ be the coordinates of a missing box in antidiagonal~$k$
(whose existence is guaranteed by condition~\ref{item:cdsk2}).
Then Lemma~\ref{lem:lij} tells us that any stable set contains
at most $k$ boxes outside $\lij{\lambda}{i}{j}$, and since all
of~$\lambda$ lies outside $\lij{\lambda}{i}{j}$, this means that
the maximum size of a stable set is~$k$.  That is, the CDS
begins with~$k$.

Conversely, suppose that the CDS begins with~$k$.  If there were
some box in antidiagonal $k-1$ missing from the diagram of~$\lambda$,
then a similar argument to the one in the previous paragraph would
imply that a stable set could have size at most $k-1$, which is a
contradiction.  So Condition~\ref{item:cdsk1} holds.
Condition~\ref{item:cdsk2} must also hold, because if the diagram
of~$\lambda$ contained the entirety of antidiagonal~$k$, then there
would be a stable set of size at least~$k+1$, again a contradiction.
\end{proof}

We use the term \emph{main antidiagonal} to refer to antidiagonal~$\delta_1-1$.
Proposition~\ref{prop:cdsk} tells us a lot about~$\delta_2$
as well.  If $\delta_1 = k$, then in the worst case,
we can color antidiagonal $k-1$ one color and
antidiagonal $k-2$ another color. Therefore, $\delta_2 \ge k-1$.
A slightly more careful argument yields the following result.

\begin{proposition}
\label{prop:catalan}
If $\delta$ is the CDS of a partition graph~$G_\lambda$
and $\delta_1 = k$, then $\delta_r \ge k-r+1$ for $1\le r\le k$.
\end{proposition}

\begin{proof}
Suppose that $\alpha_{r-1}$ boxes of the diagram of~$\lambda$
have been colored using $r-1$ colors.  It suffices to show that
at least $k-r+1$ uncolored boxes can be colored using an $r$th color.
Proposition~\ref{prop:cdsk} tells us that the diagram of~$\lambda$
contains the entirety of antidiagonal $k-1$.
That is, for $0\le i\le k-1$, the diagram of~$\lambda$
contains at least $k-i$ boxes in row~$i$.
In particular, in row $k-r$, there are at least $r$ boxes,
and therefore at least one uncolored box, since only $r-1$ colors
have been used so far.
Let us color an uncolored box in row $k-r$ with color~$r$.
By the same kind of reasoning,
there are at least two uncolored boxes in row $k-r-1$,
and at least one of these can be colored with color~$r$,
because at most one of them is rendered unavailable by
the box in row $k-r$ that we colored with color~$r$.
Continuing in this way with rows $k-r-2$, $k-r-3$, etc.,
we find that we can color one box in each of the rows
$0,1,\ldots,k-r$ with color~$r$, for a total of $k-r+1$
boxes.  This is what we set out to prove.
\end{proof}

Combining Propositions \ref{prop:deWerra} and~\ref{prop:catalan},
we see that if we fix~$\delta_1 := k$, 
then there are at most five possibilities for the first three
terms of a CDS of a partition graph:
\begin{equation*}
(k,k,k), \quad (k,k,k-1), \quad (k,k,k-2), \quad
(k,k-1,k-1), \quad (k,k-1,k-2).
\end{equation*}
Similarly, there are at most $C_r$ possibilities for the first $r$
terms, where $r$ is the $r$th Catalan number.  It turns out, however,
that not all of these possibilities are realizable.
But before proving that, we need to develop a bit of general theory.

% consider the ``degenerate'' case in which $\mu$ is the empty partition.
% For $\lambda$ to include~$\mu$ at~$d$ means that there
% exists $(i,j)$ in antidiagonal~$d$ that is missing from the
% diagram of~$\lambda$.
% So the conditions in Proposition~\ref{prop:cdsk}
% can be rephrased as, ``$\lambda$ excludes the empty partition
% at $k-1$ and includes the empty partition at~$k$.''

% For a less trivial example, consider what it means for $\lambda$ to include
% the partition $(2,1)$ at~$d$. This means that there exists $(i,j)$ in antidiagonal
% $d$ such that $\lij{\lambda}{i}{j} \simeq (2,1)$.
% This is the same as saying
% that there are three ``consecutive'' boxes in antidiagonal~$d+2$
% that are missing from~$\lambda$
% (where ``consecutive'' means that they touch at a corner).

\subsection{Corner Constraints}

Lemma~\ref{lem:lij} can be used to derive what we call
\emph{corner constraints} on the CDS of a partition graph.
First, we need a definition.

\begin{definition}
\label{def:include}
Let $\lambda$ and $\mu$ be partitions,
and let $d$ be a nonnegative integer.
Then \emph{$\lambda$ includes~$\mu$ at~$d$}
if there exist nonnegative integers $i$ and~$j$ such that
\begin{enumerate}
\item $i+j = d$, and
\item $\lij{\lambda}{i}{j} \simeq \mu$ (i.e., the diagram of~$\mu$
is isomorphic to $\lij{\lambda}{i}{j}$).
\end{enumerate}
If $\lambda$ does not include~$\mu$ at~$d$, then we say that
$\lambda$ \emph{excludes} $\mu$ at~$d$.
\end{definition}

As an example,
Definition~\ref{def:include} allows us to rephrase
Proposition~\ref{prop:cdsk} as follows:
$\delta_1 = k$ if and only if $\lambda$
excludes the empty partition at $d-1$
and includes the empty partition at~$d$.

\begin{lemma}[corner constraints]
\label{lem:include}
If $\lambda$ includes $\mu$ at~$d$, then for all $r\ge 1$,
\begin{equation}
\label{eq:include}
\alpha_r(G_\lambda) \le rd + |\mu|.
\end{equation}
\end{lemma}

\begin{proof}
By Lemma~\ref{lem:lij}, every stable set contains at most $d$
elements outside $\lij{\lambda}{i}{j}$, so the union of $r$ stable sets
contains at most $rd$ elements \emph{outside} $\lij{\lambda}{i}{j}$.
On the other hand,
since $\lij{\lambda}{i}{j} \simeq \mu$, the union of $r$ stable sets
contains at most $|\mu|$ elements \emph{inside} $\lij{\lambda}{i}{j}$.
Combining these two facts yield Inequality~\eqref{eq:include}.
\end{proof}

A remarkable consequence of matching theory is that corner constraints
actually determine the values of~$\alpha_r$, in the following sense.

\begin{theorem}
\label{thm:corner}
The value of $\alpha_r(G_\lambda)$ is the minimum value of
$rd + |\mu|$ as $i$ and~$j$ range over all nonnegative integers,
where $d := i+j$ and $\mu := \lij{\lambda}{i}{j}$.
\end{theorem}

\begin{proof}
It is a standard fact from matching theory~\cite[Volume~A, Corollary 21.4b]{S}
that if $B$ is a bipartite graph with vertex set~$V$,
then the maximum size of the union
of $r$ matchings is the minimum value of
\begin{equation*}
r\cdot |X| + |E[V \backslash X]|
\end{equation*}
as $X$ ranges over all subsets of~$V$, and $|E[V \backslash X]|$ is
the number of edges, neither of whose endvertices lie in~$X$.
In the context of~$B_\lambda$, this means that we should consider
all subsets of rows and columns---let $i$ denote the number of rows
and let $j$ denote the number of columns---and minimize over
$r\cdot (i+j)$ plus the number of boxes
not in any of the chosen rows or columns.
The key observation is that since $\lambda$ is a partition,
we can safely restrict attention to the case in which the chosen rows
and columns are the \emph{first} $i$ rows and the \emph{first} $j$ columns,
because choosing any other $i$ rows or $j$ columns can only
increase the value of $|E[V\backslash X]|$
(while keeping the value of $|X| = i + j$ the same), thereby yielding
a weaker upper bound on the value of~$\alpha_r$.
But if we choose the first $i$ rows and
the first $j$ columns, then the inequality we obtain is just a
corner constraint.
\end{proof}

\subsection{Combinatorial Characterization of the CDS}

\relax From Theorem~\ref{thm:corner}, we can infer an algorithm for
computing the CDS of a partition graph~$G_\lambda$.
At first glance, it seems that we must consider infinitely many
non\-negative integers $i$ and~$j$,
but it turns out that we need only consider finitely many pairs $(i,j)$.
To see this, first find $(i_0,j_0)$ not in the diagram of~$\lambda$
such that $i_0 + j_0$ is as small as possible;
by Proposition~\ref{prop:cdsk}, $\delta_1 = i_0+j_0$,
so the associated corner constraints are $\alpha_r \le r \delta_1$.
No other corner constraints with $\mu = \varnothing$
need to be considered, because they all involve larger
values of~$d$ and are therefore weaker.
So to compute the CDS of~$G_\lambda$,
the only other pairs $(i,j)$ we need to consider are
those that are inside the diagram of~$\lambda$.

The above algorithm is nice, but we can say more.

\begin{definition}
If $\delta$ is the CDS of~$\lambda$, then the
\emph{normalized CDS} $\bar\delta = (\bar\delta_1, \bar\delta_2, \ldots)$
is defined by $\bar\delta_i := \delta_1 - \delta_i$ for all~$i$.
\end{definition}

Knowing the CDS is equivalent to knowing the normalized CDS
plus the value of~$\delta_1$.
This may seem like a trivial change of variables,
but the advantage of the normalized CDS
is that for a fixed~$r$,
there are only finitely many possible values that the sequence
$(\bar\delta_1, \bar\delta_2, \ldots, \bar\delta_r)$ can take,
and as we now show, we can give a combinatorial characterization
of each such sequence.

We begin by rephrasing the corner constraints in terms of the normalized CDS.
Since $\alpha_r = \delta_1 + \cdots + \delta_r$,
we can subtract both sides of
the corner constraint $\alpha_r \le rd + |\mu|$ from $r\delta_1$
to obtain the equivalent inequality
\begin{equation}
\label{eq:normalizedCDS}
\bar\delta_1 + \bar\delta_2 + \cdots + \bar\delta_r
  \ge r(\delta_1 - d) - |\mu|.
\end{equation}
We can now rephrase the algorithm in the first paragraph of this subsection
in terms of the normalized CDS.
First we compute $\delta_1$, and we also note that $\bar\delta_1 = 0$ always.
We can then compute $\bar\delta_r$ by induction;
if we have already computed $\bar\delta_1, \ldots, \bar\delta_{r-1}$,
then Theorem~\ref{thm:corner} tells us that
$\bar\delta_r$ is the smallest nonnegative integer consistent with
all constraints of the form of Inequality~\eqref{eq:normalizedCDS}.
Notice that for this calculation, we just need to know the ordered pairs
$(\lij{\lambda}{i}{j}, \delta_1-i-j)$ as $(i,j)$ ranges over the diagram.
This observation is important enough that we introduce the following
definition.

\begin{definition}
\label{def:inclusionset}
The \emph{inclusion set} of a partition~$\lambda$ is the set
of all pairs $(\mu,\delta_1 - d)$ such that $\lambda$ includes $\mu$ at~$d$
and $\mu \ne \varnothing$.
\end{definition}

If we forget $\lambda$ itself but remember the inclusion set of~$\lambda$
and the value of $\delta_1$, then we can recover the normalized CDS.
In fact, to compute the normalized CDS, we do not even need the
entire inclusion set.
Any $(\mu,\delta_1-d)$ with $\delta_1 - d\le 0$ can be safely omitted from
the inclusion set, because then Inequality~\eqref{eq:normalizedCDS}
is already implied by the nonnegativity of the~$\bar\delta_i$.

Now let us fix~$r$ and think about how to compute
the first $r$ terms of the normalized CDS.
We can omit any $(\mu,\delta_1-d)$ with $|\mu| \ge r(\delta_1-d)$, because 
again Inequality~\eqref{eq:normalizedCDS}
is then implied by the nonnegativity of the~$\bar\delta_i$.
We also have the following result.

\begin{lemma}
\label{lem:redundant}
If we delete every $(\mu,\delta_1-d)$ with $\delta_1 - d\ge r$ from
the inclusion set of~$\lambda$, then the remaining elements
still suffice to recover $\bar\delta_1, \ldots, \bar\delta_r$.
\end{lemma}

\begin{proof}
We may assume that $r\ge 2$ since $\bar\delta_1 = 0$ always.
If $(\mu,\delta_1-d)$ is in the inclusion set, then there exists $(i,j)$
such that $i+j=d$ and $\lij{\lambda}{i}{j} \simeq \mu$.
Proposition~\ref{prop:cdsk} tells us that
the diagram of~$\lambda$ contains all of antidiagonal $\delta_1 - 1$.
By assumption, $i+j = d\le \delta_1 - r$, so
$\lij{\lambda}{i}{j}$ must contain all the boxes in row~$i$ extending from
$(i,j)$ out to antidiagonal $\delta_1 - 1$, namely the boxes
$(i,j), (i,j+1), \ldots, (i,j+r-1)$.
Therefore, the first part of~$\mu$ is at least~$r$.

Let $\nu$ be the partition obtained from $\mu$
by deleting the first part of~$\mu$, so that
$\nu \simeq \lij{\lambda}{i+1}{j}$. 
Since $r\ge 2$ by assumption, $\nu \ne\varnothing$.
As we just noted, the first part of~$\mu$ is at least~$r$, so
\begin{equation}
\label{eq:nu}
|\nu| \le |\mu| - r.
\end{equation}
Equation~\eqref{eq:nu}, together with
the fact that $\lambda$ includes $\nu$ at $d+1$, implies that
\begin{equation*}
\bar\delta_1 + \cdots + \bar\delta_r \ge r(\delta_1 - d -1 ) - |\nu|
\ge r(\delta_1 - d) - |\mu|,
\end{equation*}
which is the constraint associated with $(\mu,\delta_1-d)$.
So omitting $(\mu,\delta_1-d)$ from the inclusion set does no harm,
since the inclusion of $\nu$ at $d+1$
implies a constraint that is at least as strong.
\end{proof}

In summary, to compute $\bar\delta_1, \ldots, \bar\delta_r$,
the only $(\mu,\delta_1-d)$ in the inclusion set
that we need are those satisfying $0 < \delta_1-d < r$,
and we need only consider $\mu$ such that
\begin{equation*}
|\mu| < r(\delta_1 - d) < r^2.
\end{equation*}
These bounds have an important consequence:
If we let $S_r$ be the set  of all pairs $(\mu,s)$
where $0 < s < r$ and $|\mu| < r^2$, then $S_r$ is finite,
and the first $r$ terms of the normalized CDS of any~$\lambda$,
\emph{no matter how large},
can be computed from a restriction of
the inclusion set of~$\lambda$ to a subset of~$S_r$.
In particular, suppose we want to characterize
all the partition graphs whose normalized CDS begins
with a specific sequence of $r$ numbers.
Even though this is an infinite set of partition graphs,
there are only finitely many possibilities for how the
inclusion set of~$\lambda$ can restrict to a subset of~$S_r$,
and this restriction determines the first $r$ terms
of the normalized CDS.
Therefore, we can prove results such as
Propositions \ref{prop:characterization2}
and~\ref{prop:characterization3} below
just by carrying out a finite computation
(exhausting over all possible subsets of~$S_r$).
Since $\bar\delta_1 = 0$ always,
we usually suppress mention of~$\bar\delta_1$.

\begin{proposition}
\label{prop:characterization2}
If $\lambda$ is a partition, then
\begin{enumerate}[(a)]
\item $\bar\delta_2 = 1$ if and only if
$\lambda$ includes $(1)$ at $\delta_1 - 1$, and
\item $\bar\delta_2 = 0$ if and only if
$\lambda$ excludes $(1)$ at $\delta_1 - 1$.
\end{enumerate}
\end{proposition}

\begin{proposition}
\label{prop:characterization3}
If $\lambda$ is a partition, then
\begin{enumerate}[(a)]
\item $(\bar\delta_2, \bar\delta_3) = (1,2)$ if and only if
$\lambda$ includes~$(1)$ at~$\delta_1-1$ and $(2,1)$ at~$\delta_1-2$;
\item $(\bar\delta_2, \bar\delta_3) = (1,1)$ if and only if
$\lambda$ includes~$(1)$ at~$\delta_1-1$ and excludes $(2,1)$ at~$\delta_1-2$;
\item $(\bar\delta_2, \bar\delta_3) = (0,2)$ if and only if
$\lambda$ excludes~$(1)$ at~$\delta_1-1$ and includes $(2,2)$ at~$\delta_1-2$;
\item $(\bar\delta_2, \bar\delta_3) = (0,1)$ if and only if
$\lambda$ excludes~$(1)$ at~$\delta_1-1$ and $(2,2)$ at~$\delta_1-2$
and includes $(2)$ or $(1,1)$ at~$\delta_1-1$;
\item $(\bar\delta_2, \bar\delta_3) = (0,0)$ if and only if
$\lambda$ excludes~$(1)$ and $(2)$ and $(1,1)$ at~$\delta_1-1$
and $(2,2)$ at~$\delta_1-2$.
\end{enumerate}
\end{proposition}

We have computed similar characterizations
of $(\bar\delta_2, \ldots, \bar\delta_r)$
up to $r=7$, but the corresponding Propositions become increasingly
long and complicated, so we have chosen not to state them explicitly.
But there is one remark we should make.
As we noted earlier, there are at most Catalan-many possible
values for the first $r$ terms of the normalized CDS of a partition.
The first surprise is that for $r=4$, there are only 13, and not~14,
possibilities; it turns out that for no partition graph 
do we have $(\bar\delta_2,\bar\delta_3,\bar\delta_4) = (0,2,3)$.
That is not all; for example, 
$(\bar\delta_2,\bar\delta_3,\bar\delta_4,\bar\delta_5,\bar\delta_6)$
cannot assume any of the following values:
$(0,0,0,2,5)$, $(0,0,0,3,5)$, $(0,0,0,4,5)$, $(0,0,1,3,5)$,
$(0,0,1,4,5)$, $(0,0,2,3,5)$, $(0,0,2,4,4)$, $(0,0,2,4,5)$,
$(0,0,3,3,4)$, $(0,0,3,3,5)$, $(0,0,3,4,4)$, $(0,0,3,4,5)$,
$(0,1,1,4,5)$, $(0,1,2,4,5)$, $(0,1,3,3,5)$, $(0,1,3,4,4)$,
$(0,1,3,4,5)$, $(0,2,2,4,5)$, $(1,1,3,4,5)$,
as well as of course anything starting with $(0,2,3)$.
The number of possible 
values for the first $r$ terms of the normalized CDS of a partition
is an integer sequence that begins
\begin{equation*}
1, 2, 5, 13, 37, 108, 334, \ldots
\end{equation*}
The closest match to this sequence in the Online Encyclopedia of
Integer Sequences (\texttt{oeis.org}) as of this writing is A036249,
the number of rooted trees of nonempty sets with $n$ points, which begins
\begin{equation*}
1, 2, 5, 13, 37, 108, 332, \ldots
\end{equation*}

\section{Partial CDS-Colorability}

We mentioned in the Introduction that we have been able to prove
the existence of a coloring that matches the first four terms
of the CDS. We give the proof in this section.

\begin{definition}
\label{def:Ln}
Let $\LL_n$ denote the statement that
for every partition graph~$G$, there exists a coloring~$\kappa$
such that the first $n$ parts of $\shape(\kappa)$ coincide
with the first $n$ terms of the CDS of~$G$.
\end{definition}

The scrupulous reader may worry about what Definition~\ref{def:Ln}
means if the length of $\shape(\kappa)$ or
of the CDS is less than~$n$.
For the purposes of Definition~\ref{def:Ln},
in such cases we simply ``pad out'' the sequences with trailing zeros.
So in particular, if $\LL_n$ is true for all~$n$,
then the Latin Tableau Conjecture is true.

The statement~$\LL_1$ becomes obvious as soon as one
unwinds the definitions.  The statements $\LL_2$ and~$\LL_3$
are not too difficult either, as we now show.

\begin{theorem}
\label{thm:L2L3}
$\LL_2$ and $\LL_3$ are true.
\end{theorem}

\begin{proof}
We know from Propositions \ref{prop:deWerra} and~\ref{prop:catalan}
that the only possibilities for $\delta_2$ are $\delta_2 = \delta_1$
and $\delta_2 = \delta_1 - 1$.
Suppose first that $\delta_2 = \delta_1$.
That means that $\alpha_2 = 2\delta_1$, which means that
there exists a disjoint union of two stable sets whose cardinality
is $2\delta_1$. But by definition of~$\delta_1$, each stable
set has cardinality at most~$\delta_1$. So each of the two
stable sets has cardinality $\delta_1$, and we have exhibited
a coloring with the desired properties.
Similarly, if $\delta_2 = \delta_1 - 1$, then there
must exist a union of two disjoint stable sets whose cardinality
is $2\delta_1 - 1$. But the only way this can be realized
is if one stable set has size $\delta_1$ and the other
has size $\delta_1 - 1$. The proof of~$\LL_2$ is complete.

There are five possibilities for the first three terms of the CDS.
A similar argument to the one in the previous paragraph handles
the cases $\delta_1 = \delta_2 = \delta_3$ and
$\delta_1 = \delta_2$, $\delta_3 = \delta_1 - 1$.
If $\delta_3 = \delta_2 = \delta_1 - 1$, then there exists
a union of three disjoint stable sets whose cardinality is $3\delta_1 - 2$,
but because $\alpha_2 = \delta_1 + \delta_2 = 2\delta_1 - 1$,
at most one of these stable sets can have cardinality~$\delta_1$.
So the only possibility is that we have one stable set of
size~$\delta_1$ and two stable sets of size $\delta_1-1$,
which is what we need.

In the remaining two cases, $\delta_3 = \delta_1 - 2$.
Since $\LL_2$ is true, we know that there exist two stable sets
with cardinalities $\delta_1$ and~$\delta_2$ respectively.
We can now establish the existence of a third stable set,
disjoint from the first two and having cardinality $\delta_1-2$,
via the same kind of argument that we used in the proof of
Proposition~\ref{prop:catalan}.
Namely, the diagram of~$\lambda$ contains the entirety of
antidiagonal $\delta_1 - 1$,
and so there are at least $3$ boxes in row $\delta_1 - 3$,
one of which we can use for our third stable set.
We can work our way up the diagram of~$\lambda$
one row at a time, adding one box to our stable set
in each row.  This completes the proof of~$\LL_3$.
\end{proof}

The proof of $\LL_4$ is of course more involved.
We can handle all but five cases using ideas similar 
to those in the proof of Theorem~\ref{thm:L2L3}.
For convenience, we state this partial result as
a separate Proposition.

\begin{proposition}
\label{prop:L4}
$\LL_4$ is true except possibly when
$(\bar\delta_2,\bar\delta_3,\bar\delta_4)$
equals $(0,0,2)$ or $(0,1,2)$ or $(0,2,2)$ or
$(1,1,2)$ or $(1,2,2)$.
\end{proposition}

\begin{proof}
Let us first dispense with the case $\bar\delta_4 = 3$.
Theorem~\ref{thm:L2L3} tells us that $\LL_3$ is true,
so there exist three disjoint stable sets with
cardinalities $\delta_1, \delta_2, \delta_3$ respectively.
The diagram of~$\lambda$ contains the entire main antidiagonal,
so there are at least 4~boxes in row $\delta_1 - 4$,
one of which we can use for our fourth stable set.
We can work our way up the diagram of~$\lambda$
one row at a time, adding one box to our fourth stable set
in each row. This gives us a stable set of size
$\delta_1 - 3$, so indeed $\bar\delta_4 = 3$.

If $(\bar\delta_2,\bar\delta_3,\bar\delta_4)$
is $(0,0,0)$ (respectively, $(0,0,1)$),
then the only way to cover $4\delta_1$ (respectively, $4\delta_1-1$)
boxes with four stable sets, each of size at most $\delta_1$,
is for the stable sets to have sizes $\delta_1, \delta_1, \delta_1, \delta_1$
(respectively, $\delta_1, \delta_1, \delta_1, \delta_1-1$),
which is precisely what we want.

Suppose $(\bar\delta_2,\bar\delta_3,\bar\delta_4) = (1,1,1)$.
Then at most one color can occur $\delta_1$ times, so the only way to
cover $4\delta_1 - 3$ boxes with four stable sets is with stable sets of
sizes $\delta_1, \delta_1-1, \delta_1-1, \delta_1-1$.
Finally, if $(\bar\delta_2,\bar\delta_3,\bar\delta_4) = (0,1,1)$,
then at most two colors can occur $\delta_1$ times, so the only way to cover
$4\delta_1 - 2$ boxes with four stable sets is with stable sets of
sizes $\delta_1, \delta_1, \delta_1-1, \delta_1-1$.
\end{proof}

For the five outstanding cases, we can deduce some necessary consequences.

\begin{proposition}
\label{prop:fivecases}
Let $\lambda$ be a partition.
\begin{enumerate}[(a)]
\item \label{item:case002}
If $(\bar\delta_2, \bar\delta_3, \bar\delta_4) = (0,0,2)$
then $\lambda$ excludes $(1)$, $(1,1)$, and~$(2)$ at $\delta_1-1$
and excludes $(3,3,3)$ at $\delta_1-3$.
\item \label{item:case012}
If $(\bar\delta_2, \bar\delta_3, \bar\delta_4) = (0,1,2)$
then $\lambda$ excludes $(1)$ at $\delta_1-1$,
excludes $(2,2)$ at $\delta_1-2$, and excludes
$(3,3,2)$ at $\delta_1-3$.
\item \label{item:case022}
If $(\bar\delta_2, \bar\delta_3, \bar\delta_4) = (0,2,2)$
then $\lambda$ excludes $(1)$ at $\delta_1-1$.
\item \label{item:case112}
If $(\bar\delta_2, \bar\delta_3, \bar\delta_4) = (1,1,2)$
then $\lambda$ excludes $(2,1)$ at $\delta_1-2$,
and excludes $(3,3,1)$ and $(3,2,2)$ at $\delta_1-3$.
\item \label{item:case122}
If $(\bar\delta_2, \bar\delta_3, \bar\delta_4) = (1,2,2)$
then $\lambda$ excludes $(3,2,1)$ at $\delta_1-3$.
\end{enumerate}
\end{proposition}

\begin{proof}
The proof is a routine verification.
For example, if $(\bar\delta_2, \bar\delta_3, \bar\delta_4) = (0,0,2)$
then $\lambda$ must exclude~$(3,3,3)$ at $\delta_1-3$,
because otherwise Equation~\eqref{eq:normalizedCDS} with $r=4$
would imply
\begin{equation*}
2 = \bar\delta_2 + \bar\delta_3 +\bar\delta_4
\ge 4\bigl(\delta_1 - (\delta_1 - 3)\bigr) - 9 = 3,
\end{equation*}
which is false. All the other assertions are easily confirmed
using the same kind of argument; we omit the details.
\end{proof}

The way they are currently stated,
the exclusion conditions in Proposition~\ref{prop:fivecases}
are not so easy to grasp, but fortunately,
in several cases, there is an easier way to think about them.
For the remainder of this section, we shall use the notation
$(i_1,j_1), (i_2,j_2), \ldots, (i_s,j_s)$
(where $i_1 < i_2 < \cdots < i_s$)
to denote the coordinates of all the boxes along
antidiagonal~$\delta_1$ that are \emph{not} in the diagram of~$\lambda$
(at least one such box must exist, by Proposition~\ref{prop:cdsk}).

As an illustration of why this notation is useful,
suppose that $\lambda$ excludes $(1)$ at $\delta_1-1$.
Then $j_{p+1} \le j_p - 2$; i.e., there cannot be
two consecutive missing boxes in antidiagonal~$\delta_1$.
See Figure~\ref{fig:2missing},
where for clarity, we have put dots in the boxes on
the main antidiagonal;
the two consecutive missing boxes $(1,3)$ and $(2,2)$
in antidiagonal~4 force $\lij{\lambda}{1}{2} \simeq (1)$,
as indicated by the red box,
so $\lambda$ includes $(1)$ at~3.
It is intuitively easier to grasp ``no two consecutive
missing boxes in antidiagonal~$\delta_1$'' than
``$\lambda$ excludes $(1)$ at~$\delta_1-1$.''

\begin{figure}[ht]
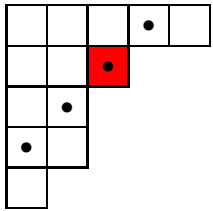

\begin{center}
\ytableausetup{centertableaux}
\begin{ytableau}
\  & \  & \ & \bullet & \  \\
\  & \  & *(red) \bullet \\
\  & \bullet \\
\bullet & \  \\
\  
\end{ytableau}
\end{center}
\caption{Two consecutive missing boxes in antidiagonal 4}
\label{fig:2missing}
\end{figure}

This type of reasoning will play an important role in our proof
of Theorem~\ref{thm:L4}.

\begin{theorem}
\label{thm:L4}
$\LL_4$ is true.
\end{theorem}

\begin{proof}
Proposition~\ref{prop:L4} leaves five remaining cases,
which we prove in approximately increasing order of complexity.
In each case, we let $\lambda$ be a partition whose normalized CDS
begins with the specified numbers,
and we construct disjoint stable sets $S_1$, $S_2$, $S_3$, $S_4$
of cardinalities $\delta_1$, $\delta_2$, $\delta_3$, $\delta_4$ 
respectively. We sometimes refer to the inclusion of a box in~$S_i$
as \emph{coloring} the box with color~$i$.
In our figures below, the presence of~$i$ in a box
means that the box is in~$S_i$.

\textbf{Case $(\bar\delta_2, \bar\delta_3, \bar\delta_4) = (1,2,2)$.}
We claim that there cannot be
four consecutive missing boxes in antidiagonal~$\delta_1$;
for if there were, then
by drawing a diagram analogous to Figure~\ref{fig:2missing},
we would deduce that 
$\lambda$ includes $(3,2,1)$ at $\delta_1-3$, but this is ruled out
by Proposition~\ref{prop:fivecases}\ref{item:case122}.
For $i=1,2,3$, we let $S_i$ comprise all the boxes on
antidiagonal $\delta_1 - i$.
We let $S_4$ comprise any and all
boxes on antidiagonal $\delta_1$ in the diagram of~$\lambda$,
as well as all boxes
with coordinates $(i_p,j_{p+3})$ for $1 \le p \le s-3$
(by Proposition~\ref{prop:characterization3},
$\lambda$ includes $(2,1)$ at $\delta_1-2$, so $s\ge 3$).
See Figure~\ref{fig:case122} for an example.

\begin{figure}[ht]
\begin{center}
\ytableausetup{smalltableaux}
\begin{ytableau}
\ & \ & \ & \ & \ & 3 & 2 & 1 & 4\\
\ & \ & \ & 4 & 3 & 2 & 1    \\
\ & 4 & \ & 3 & 2 & 1     \\
4 & \ & 3 & 2 & 1     \\
\ & 3 & 2 & 1 & 4 \\
3 & 2 & 1     \\
2 & 1 & 4    \\
1  
\end{ytableau}
\end{center}
\caption{Example for Case $(\bar\delta_2, \bar\delta_3, \bar\delta_4) = (1,2,2)$}
\label{fig:case122}
\end{figure}
There are two things to check: first, that $S_4$ has
size $\delta_1 - 2$, and second, that $S_4$
really is a stable set that is disjoint from the other three.
There are $\delta_1 + 1$ boxes on antidiagonal~$\delta_1$,
of which $s$ are not in the diagram of~$\lambda$,
and $\delta_1 + 1 - s$ are in~$S_4$.
So the cardinality of~$S_4$ is $(\delta_1 + 1 - s) + (s - 3) = \delta_1 - 2$,
as claimed. The row and column coordinates of the boxes of~$S_4$
are all distinct by construction, so $S_4$ is a stable set.
The crucial point is that $j_{p+3} \le j_p - 4$
because there cannot be four consecutive missing boxes
in antidiagonal~$\delta_1$. Therefore,
\begin{equation*}
i_p + j_{p+3} \le i_p + j_p - 4 = \delta_1-4,
\end{equation*}
and thus $(i_p,j_{p+3})$ does not lie on antidiagonals
$\delta_1-1$, $\delta_1-2$, or $\delta_1-3$.
So $S_4$ is indeed disjoint from the other three stable sets.

\textbf{Case $(\bar\delta_2, \bar\delta_3, \bar\delta_4) = (1,1,2)$.}
By Proposition~\ref{prop:characterization3},
$\lambda$ includes $(1)$ at $\delta_1-1$, so $s\ge 2$.
We claim that $j_{p+2} \le j_p - 4$ for $1\le p \le s-2$.
For if $j_{p+2} \ge j_p - 3$, then three out of four consecutive
boxes on antidiagonal~$\delta_1$ must be missing,
but by Proposition~\ref{prop:fivecases}\ref{item:case112},
$\lambda$ excludes $(2,1)$ at $\delta_1-2$, so
no three consecutive boxes on antidiagonal~$\delta_1$ are missing.
So somewhere in $\lambda$, one of the configurations
in Figure~\ref{fig:3of4} must exist
(where we have put dots in the boxes along the main antidiagonal).
That is, $\lambda$ must include either $(3,3,1)$ or $(3,2,2)$
at $\delta_1-3$, but this contradicts
Proposition~\ref{prop:fivecases}\ref{item:case112}.
\begin{figure}[ht]
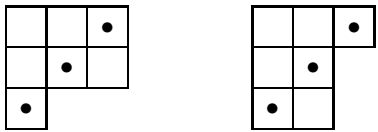

\begin{center}
\ytableausetup{nosmalltableaux}
\begin{ytableau}
\ & \ & \bullet   \\
\ & \bullet & \   \\
\bullet          \\
\end{ytableau}
$\qquad\qquad$
\begin{ytableau}
\ & \ & \bullet   \\
\ & \bullet   \\
\bullet & \         \\
\end{ytableau}
\end{center}
\caption{Missing 3 out of 4 consecutive boxes in antidiagonal~$\delta_1$}
\label{fig:3of4}
\end{figure}

For $i=1,2$, we let $S_i$ comprise all the boxes on
antidiagonal $\delta_1 -i$,
and we let $S_4$ comprise all the boxes
on antidiagonal $\delta_1 - 3$.
We let $S_3$ comprise any and all
boxes on antidiagonal $\delta_1$, as well as all boxes
with coordinates $(i_p,j_{p+2})$ for $1 \le p \le s-2$.
The rest of the argument is similar to 
the previous case; we readily verify that the cardinality
of~$S_3$ is $\delta_1 - 1$,
and because $j_{p+2} \le j_p - 4$, it follows that the
boxes with coordinates $(i_p,j_{p+2})$ lie in antidiagonal
$\delta_1 - 4$ or lower, and hence are disjoint from
$S_1$, $S_2$, and~$S_4$.
See Figure~\ref{fig:case112} for an example.
This completes the proof of
Case $(\bar\delta_2, \bar\delta_3, \bar\delta_4) = (1,1,2)$.
\begin{figure}[ht]
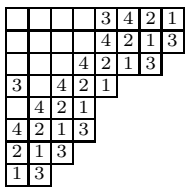

\begin{center}
\ytableausetup{smalltableaux}
\begin{ytableau}
\ & \ & \ & \ & 3 & 4 & 2 & 1  \\
\ & \ & \ & \ & 4 & 2 & 1 & 3 \\
\ & \ & \ & 4 & 2 & 1 & 3 \\
3 & \ & 4 & 2 & 1  \\
\ & 4 & 2 & 1 \\
4 & 2 & 1 & 3 \\
2 & 1 & 3 \\
1 & 3
\end{ytableau}
\end{center}
\caption{Example for Case $(\bar\delta_2, \bar\delta_3, \bar\delta_4) = (1,1,2)$}
\label{fig:case112}
\end{figure}

The remaining three cases rely on the decomposition of~$\lambda$
into what we call \emph{building blocks}.
We regard the missing boxes $(i_1,j_1), \ldots, (i_s,j_s)$
as dividing up $\lambda$ into smaller (nonempty) pieces.
More precisely, building block~$0$
is the set of boxes $(i,j)$ in~$\lambda$ with $i<i_1$ and $j_1 \le j$
(as long as this set is not empty),
building block~$1$ comprises the boxes $(i,j)$ in~$\lambda$
with $i_1 \le i < i_2$ and $j_2 \le j < j_1$, and so on.
See Figure~\ref{fig:buildingblocks} for an example,
where the building blocks are colored green
(and dots are placed along the main antidiagonal for reference).
Along antidiagonal~$\delta_1$, a building block contains all the boxes
except that it omits the box in column~0, or the box in row~0, or both;
we call these three kinds of building blocks \emph{top},
\emph{bottom}, and \emph{interior} blocks respectively.
In Figure~\ref{fig:buildingblocks}, there are three interior blocks,
and a bottom block, but no top block.

\begin{figure}[ht]
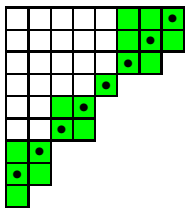

\begin{center}
\begin{ytableau}
\ & \ & \ & \ & \ & *(green) & *(green) & *(green) \bullet \\
\ & \ & \ & \ & \ & *(green) & *(green) \bullet & *(green) \\
\ & \ & \ & \ & \ & *(green) \bullet & *(green) \\
\ & \ & \ & \ & *(green) \bullet  \\
\ & \ & *(green) & *(green) \bullet  \\
\ & \ & *(green) \bullet & *(green) \\
*(green) & *(green) \bullet  \\
*(green) \bullet & *(green) \\
*(green)  \\
\end{ytableau}
\end{center}
\caption{Building Blocks}
\label{fig:buildingblocks}
\end{figure}

The construction of $S_1, \ldots, S_4$ proceeds in two phases.
In Phase~1, we show how to color individual building blocks.
Actually, for each kind of building block (top, bottom, interior),
we show only how to color the \emph{minimal} building blocks
of a given height; minimality means that any other building block
(for partitions with the specified normalized CDS values)
of the same kind and height contains a minimal one.
It suffices to show how to color minimal building blocks;
if a building block is not minimal,
we can simply ignore the extra boxes when forming $S_1, \ldots, S_4$.
In Phase~2, we color~$\lambda$ by coloring each building block
the way it was colored in Phase~1,
and then coloring additional boxes as needed.

There is one important subtlety.
Let $\Lambda(a,b,c)$ denote the set of partitions such that
$(\bar\delta_2, \bar\delta_3, \bar\delta_4) = (a,b,c)$.
A building block for $\Lambda(a,b,c)$ is often, but not always,
itself a member of $\Lambda(a,b,c)$.
For example, $(1,1,1)$ is a top block for $\Lambda(0,0,2)$
but $(1,1,1) \notin \Lambda(0,0,2)$.
To check that our colorings are valid,
we must verify that the building blocks that are themselves in $\Lambda(a,b,c)$
are correctly colored in Phase~1,
and we must also verify that Phase~2
produces correct colorings whether or not the constituent building blocks
are themselves in $\Lambda(a,b,c)$.

\textbf{Case $(\bar\delta_2, \bar\delta_3, \bar\delta_4) = (0,0,2)$.}
Proposition~\ref{prop:fivecases}\ref{item:case002} says that $\lambda$ excludes
$(1)$ at $\delta_1-1$ (which as we noted above means that there
are no two consecutive missing boxes in antidiagonal~$\delta_1$),
and that $\lambda$ excludes $(1,1)$ and $(2)$ at $\delta_1-1$.
Therefore, $j_{p+1} \le j_p - 3$;
i.e., the missing boxes in antidiagonal~$\delta_1$
must be at least three apart,
and an interior building block must have a box at
the start and end of antidiagonal $\delta_1+1$.
Figures \ref{fig:002a} and~\ref{fig:002b} illustrate the
minimal interior building blocks of heights 4--12,
along with a Phase~1 coloring
(there is no interior building block of height~3
because Proposition~\ref{prop:fivecases}\ref{item:case002}
says that $\lambda$ excludes $(3,3,3)$ at $\delta_1-3$).
It should be clear from
Figure~\ref{fig:002b} how the pattern
extends to minimal interior building blocks with larger height.

\begin{figure}[ht]
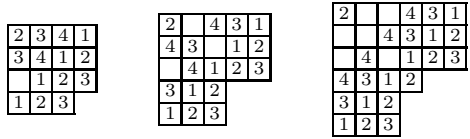

\begin{center}
\ytableausetup{smalltableaux}
\begin{ytableau}
2 & 3 & 4 & 1 \\
3 & 4 & 1 & 2 \\
\ & 1 & 2 & 3 \\
1 & 2 & 3
\end{ytableau}
\qquad
\begin{ytableau}
2 & \ & 4 & 3 & 1 \\
4 & 3 & \ & 1 & 2 \\
\ & 4 & 1 & 2 & 3 \\
3 & 1 & 2 \\
1 & 2 & 3
\end{ytableau}
\qquad
\begin{ytableau}
2 & \ & \ & 4 & 3 & 1 \\
\ & \ & 4 & 3 & 1 & 2 \\
\ & 4 & \ & 1 & 2 & 3 \\
4 & 3 & 1 & 2 \\
3 & 1 & 2 \\
1 & 2 & 3
\end{ytableau}
\end{center}
\caption{Minimal Interior Building Blocks of Heights 4, 5, 6}
\label{fig:002a}
\end{figure}
\begin{figure}[ht]
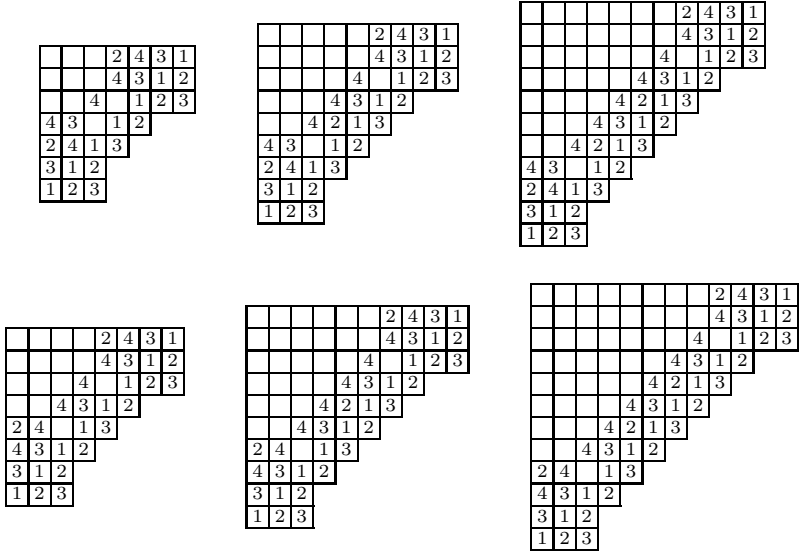

\begin{center}
\ytableausetup{smalltableaux}
\begin{ytableau}
\ & \ & \ & 2 & 4 & 3 & 1 \\
\ & \ & \ & 4 & 3 & 1 & 2 \\
\ & \ & 4 & \ & 1 & 2 & 3 \\
4 & 3 & \ & 1 & 2 \\
2 & 4 & 1 & 3 \\
3 & 1 & 2 \\
1 & 2 & 3
\end{ytableau}
\qquad
\begin{ytableau}
\ & \ & \ & \ & \ & 2 & 4 & 3 & 1 \\
\ & \ & \ & \ & \ & 4 & 3 & 1 & 2 \\
\ & \ & \ & \ & 4 & \ & 1 & 2 & 3 \\
\ & \ & \ & 4 & 3 & 1 & 2 \\
\ & \ & 4 & 2 & 1 & 3 \\
4 & 3 & \ & 1 & 2 \\
2 & 4 & 1 & 3 \\
3 & 1 & 2 \\
1 & 2 & 3
\end{ytableau}
\qquad
\begin{ytableau}
\ & \ & \ & \ & \ & \ & \ & 2 & 4 & 3 & 1 \\
\ & \ & \ & \ & \ & \ & \ & 4 & 3 & 1 & 2 \\
\ & \ & \ & \ & \ & \ & 4 & \ & 1 & 2 & 3 \\
\ & \ & \ & \ & \ & 4 & 3 & 1 & 2 \\
\ & \ & \ & \ & 4 & 2 & 1 & 3 \\
\ & \ & \ & 4 & 3 & 1 & 2 \\
\ & \ & 4 & 2 & 1 & 3 \\
4 & 3 & \ & 1 & 2 \\
2 & 4 & 1 & 3 \\
3 & 1 & 2 \\
1 & 2 & 3
\end{ytableau}
\\ \ \\ \ \\
\begin{ytableau}
\ & \ & \ & \ & 2 & 4 & 3 & 1 \\
\ & \ & \ & \ & 4 & 3 & 1 & 2 \\
\ & \ & \ & 4 & \ & 1 & 2 & 3 \\
\ & \ & 4 & 3 & 1 & 2 \\
2 & 4 & \ & 1 & 3 \\
4 & 3 & 1 & 2 \\
3 & 1 & 2 \\
1 & 2 & 3
\end{ytableau}
\qquad
\begin{ytableau}
\ & \ & \ & \ & \ & \ & 2 & 4 & 3 & 1 \\
\ & \ & \ & \ & \ & \ & 4 & 3 & 1 & 2 \\
\ & \ & \ & \ & \ & 4 & \ & 1 & 2 & 3 \\
\ & \ & \ & \ & 4 & 3 & 1 & 2 \\
\ & \ & \ & 4 & 2 & 1 & 3 \\
\ & \ & 4 & 3 & 1 & 2 \\
2 & 4 & \ & 1 & 3 \\
4 & 3 & 1 & 2 \\
3 & 1 & 2 \\
1 & 2 & 3
\end{ytableau}
\qquad
\begin{ytableau}
\ & \ & \ & \ & \ & \ & \ & \ & 2 & 4 & 3 & 1 \\
\ & \ & \ & \ & \ & \ & \ & \ & 4 & 3 & 1 & 2 \\
\ & \ & \ & \ & \ & \ & \ & 4 & \ & 1 & 2 & 3 \\
\ & \ & \ & \ & \ & \ & 4 & 3 & 1 & 2 \\
\ & \ & \ & \ & \ & 4 & 2 & 1 & 3 \\
\ & \ & \ & \ & 4 & 3 & 1 & 2 \\
\ & \ & \ & 4 & 2 & 1 & 3 \\
\ & \ & 4 & 3 & 1 & 2 \\
2 & 4 & \ & 1 & 3 \\
4 & 3 & 1 & 2 \\
3 & 1 & 2 \\
1 & 2 & 3
\end{ytableau}
\end{center}
\caption{Minimal Interior Building Blocks of Heights 7--12}
\label{fig:002b}
\end{figure}

The minimal top (respectively, bottom) building blocks,
together with their colorings,
may be obtained by deleting the top three rows
(respectively, left three columns) from
the minimal interior building blocks.
All these building blocks, except $(3)$ and $(1,1,1)$
and $(4,4,4,3)$, are in $\Lambda(0,0,2)$,
and it is readily verified that the given colorings are valid.

In Phase~2, after coloring each building block as in Phase~1,
$S_1$ will always already have the correct cardinality.
Moreover, in each building block, $S_2$ and~$S_3$ are the same
size as~$S_1$, so they also already have the correct cardinality.
Thus we just need to show that $S_4$ can be suitably enlarged.
To see this, first note that if $(3)$ and $(1,1,1)$ and $(4,4,4,3)$
are not involved, then between any two consecutive building blocks,
we can put two more~4's in antidiagonal $\delta_1-3$,
and this will make $S_4$ the right size without putting
any 4's in a row or column with an existing~4.
See Figure~\ref{fig:002concat} for an example.

\begin{figure}[ht]
\begin{center}
\begin{ytableau}
\  & \  & \  & *(green) 2 & *(green) \ & *(green) 4 & *(green) 3 & *(green) 1 \\
\  & \  & \  & *(green) 4 & *(green) 3 & *(green) \ & *(green) 1 & *(green) 2 \\
\  & \  & \  & *(green) \ & *(green) 4 & *(green) 1 & *(green) 2 & *(green) 3 \\
\  & \  & 4  & *(green) 3 & *(green) 1 & *(green) 2 \\
\  & 4  & \  & *(green) 1 & *(green) 2 & *(green) 3 \\
*(green) \ & *(green) 3 & *(green) 1 \\
*(green) 4 & *(green) 1 & *(green) 2 \\
*(green) 1 & *(green) 2 & *(green) 3 \\
*(green) 2 \\
*(green) 3
\end{ytableau}
\end{center}
\caption{Case $(\bar\delta_2, \bar\delta_3, \bar\delta_4) = (0,0,2)$:
Phase~2}
\label{fig:002concat}
\end{figure}

Finally, if any of the building blocks $(3)$ or $(1,1,1)$ or $(4,4,4,3)$
are involved,
we put a 4 in antidiagonal $\delta_1-2$ between building blocks,
and if $(4,4,4,3)$ is involved, we also put a 4 in
the one as-yet-uncolored box of $(4,4,4,3)$ if necessary.
It is easy to check that this construction works.
See Figure~\ref{fig:002concat2}.

\begin{figure}[ht]
\begin{center}
\begin{ytableau}
\ & \ & \ & \ & \ & 4 & *(green) 1 & *(green) 2 & *(green) 3 \\
\ & \ & *(green) 2 & *(green) 3 & *(green) 4 & *(green) 1 \\
\ & \ & *(green) 3 & *(green) 4 & *(green) 1 & *(green) 2 \\
\ & \ & *(green) 4 & *(green) 1 & *(green) 2 & *(green) 3 \\
\ & 4 & *(green) 1 & *(green) 2 & *(green) 3 \\
*(green) 3 & *(green) 1 \\
*(green) 1 & *(green) 2 \\
*(green) 2 & *(green) 3
\end{ytableau}
\qquad
\begin{ytableau}
\ & \ & \ & \ & *(green) 2 & *(green) 3 & *(green) 4 & *(green) 1 \\
\ & \ & \ & \ & *(green) 3 & *(green) 4 & *(green) 1 & *(green) 2 \\
\ & \ & \ & \ & *(green) 4 & *(green) 1 & *(green) 2 & *(green) 3 \\
\ & \ & \ & 4 & *(green) 1 & *(green) 2 & *(green) 3 \\
*(green) 2 & *(green) 3 & *(green) 4 & *(green) 1 \\
*(green) 3 & *(green) 4 & *(green) 1 & *(green) 2 \\
*(green) \ & *(green) 1 & *(green) 2 & *(green) 3 \\
*(green) 1 & *(green) 2 & *(green) 3
\end{ytableau}
\end{center}
\caption{Case $(\bar\delta_2, \bar\delta_3, \bar\delta_4) = (0,0,2)$:
Phase~2}
\label{fig:002concat2}
\end{figure}

\textbf{Case $(\bar\delta_2, \bar\delta_3, \bar\delta_4) = (0,1,2)$.}
Proposition~\ref{prop:fivecases}\ref{item:case012} says that
$\lambda$ excludes $(1)$ at $\delta_1 - 1$
and excludes $(2,2)$ at $\delta_1 - 2$;
by arguments similar to those we have seen before,
this implies that $j_{p+1} \le j_p-3$.
Moreover, Proposition~\ref{prop:fivecases}\ref{item:case012} says that
$\lambda$ excludes $(3,3,2)$ at $\delta_1 - 3$,
so if $j_{p+1} = j_p - 3$ then $\lij{\lambda}{i_p}{j_{p+1}} \simeq (3,3,3)$.
Therefore:
\begin{enumerate} \item \label{item:012a}
There is no restriction on the heights of minimal top and bottom building blocks,
and they contain no boxes in antidiagonals $\delta_1+1$ or higher.
\item \label{item:012c}
The shape $(3,3,3)$ is a minimal interior building block of height~3.
\item \label{item:012b}
The remaining minimal interior building blocks have height at least~4,
and they contain no boxes in antidiagonals $\delta_1+1$ or higher.
\end{enumerate}
For Phase~1,
in both Cases \ref{item:012a} and~\ref{item:012b},
we put all boxes (if any) in the diagram of~$\lambda$
on antidiagonal $\delta_1 - 1$
(respectively, $\delta_1$, $\delta_1-2$, $\delta_1-3$)
in~$S_1$ (respectively, $S_2$, $S_3$, $S_4$).
In Case~\ref{item:012b}, we also put the top left box $(0,0)$ in~$S_2$.
See Figure~\ref{fig:012building},
which also shows how we color $(3,3,3)$.

\begin{figure}[ht]
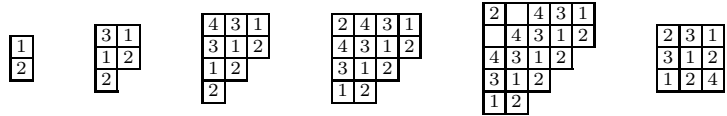

\begin{center}
\begin{ytableau}
1 \\
2
\end{ytableau}
\qquad
\begin{ytableau}
3 & 1 \\
1 & 2 \\
2
\end{ytableau}
\qquad
\begin{ytableau}
4 & 3 & 1 \\
3 & 1 & 2 \\
1 & 2 \\
2
\end{ytableau}
\qquad
\begin{ytableau}
2 & 4 & 3 & 1 \\
4 & 3 & 1 & 2 \\
3 & 1 & 2 \\
1 & 2
\end{ytableau}
\qquad
\begin{ytableau}
2 & \ & 4 & 3 & 1 \\
\ & 4 & 3 & 1 & 2 \\
4 & 3 & 1 & 2 \\
3 & 1 & 2 \\
1 & 2
\end{ytableau}
\qquad
\begin{ytableau}
2 & 3 & 1 \\
3 & 1 & 2 \\
1 & 2 & 4
\end{ytableau}
\end{center}
\caption{Case $(\bar\delta_2, \bar\delta_3, \bar\delta_4) = (0,1,2)$:
Building Blocks}
\label{fig:012building}
\end{figure}

For Phase~2, if $(3,3,3)$ is not one of the building blocks,
we simply fill antidiagonal $\delta_1-2$ with 3's and
antidiagonal $\delta_1-3$ with 4's.
If $(3,3,3)$ is involved, then see Figure~\ref{fig:012concat}.
Again, the verification that this construction is correct is
straightforward and we omit the details.
\begin{figure}[ht]
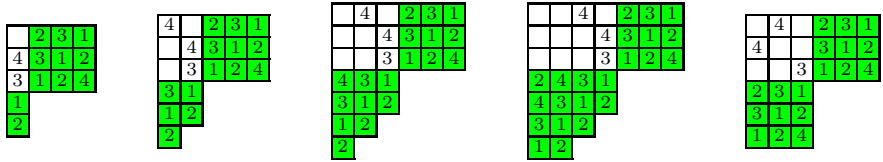

\begin{center}
\begin{ytableau}
\ & *(green) 2 & *(green) 3 & *(green) 1 \\
4 & *(green) 3 & *(green) 1 & *(green) 2 \\
3 & *(green) 1 & *(green) 2 & *(green) 4 \\
*(green) 1 \\
*(green) 2
\end{ytableau}
\qquad
\begin{ytableau}
4 & \ & *(green) 2 & *(green) 3 & *(green) 1 \\
\ & 4 & *(green) 3 & *(green) 1 & *(green) 2 \\
\ & 3 & *(green) 1 & *(green) 2 & *(green) 4 \\
*(green) 3 & *(green) 1 \\
*(green) 1 & *(green) 2 \\
*(green) 2
\end{ytableau}
\qquad
\begin{ytableau}
\ & 4 & \ & *(green) 2 & *(green) 3 & *(green) 1 \\
\ & \ & 4 & *(green) 3 & *(green) 1 & *(green) 2 \\
\ & \ & 3 & *(green) 1 & *(green) 2 & *(green) 4 \\
*(green) 4 & *(green) 3 & *(green) 1 \\
*(green) 3 & *(green) 1 & *(green) 2 \\
*(green) 1 & *(green) 2 \\
*(green) 2
\end{ytableau}
\qquad
\begin{ytableau}
\ & \ & 4 & \ & *(green) 2 & *(green) 3 & *(green) 1 \\
\ & \ & \ & 4 & *(green) 3 & *(green) 1 & *(green) 2 \\
\ & \ & \ & 3 & *(green) 1 & *(green) 2 & *(green) 4 \\
*(green) 2 & *(green) 4 & *(green) 3 & *(green) 1 \\
*(green) 4 & *(green) 3 & *(green) 1 & *(green) 2 \\
*(green) 3 & *(green) 1 & *(green) 2 \\
*(green) 1 & *(green) 2
\end{ytableau}
\qquad
\begin{ytableau}
\ & 4 & \ & *(green) 2 & *(green) 3 & *(green) 1 \\
4 & \ & \ & *(green) 3 & *(green) 1 & *(green) 2 \\
\ & \ & 3 & *(green) 1 & *(green) 2 & *(green) 4 \\
*(green) 2 & *(green) 3 & *(green) 1 \\
*(green) 3 & *(green) 1 & *(green) 2 \\
*(green) 1 & *(green) 2 & *(green) 4
\end{ytableau}
\end{center}
\caption{Case $(\bar\delta_2, \bar\delta_3, \bar\delta_4) = (0,1,2)$:
Phase~2}
\label{fig:012concat}
\end{figure}

\textbf{Case $(\bar\delta_2, \bar\delta_3, \bar\delta_4) = (0,2,2)$.}
Proposition~\ref{prop:fivecases}\ref{item:case022} says that
$\lambda$ excludes $(1)$ at $\delta_1 - 1$.
Therefore, the partition~$(1)$ is not a valid building block;
the minimal building blocks of height at most~5
are those shown in Figure~\ref{fig:022building}.
The larger minimal building blocks are interior building blocks
that contain all the boxes in antidiagonal~$\delta_1$
except the box in row~0 and the box in column~0,
and no other boxes.
(We need not separately consider top and bottom building blocks of height at least~4
because they contain a valid interior building block as a proper subset.)
To color one of these larger building blocks,
we let $S_1$ and~$S_4$ comprise all of
antidiagonals $\delta-1-1$ and $\delta_1-3$ respectively;
we let $S_3$ comprise all of antidiagonal $\delta_1-2$
except for the box in column~0,
and we let $S_2$ comprise all of antidiagonal~$\delta_1$
plus the box with coordinates $(0,0)$.
\begin{figure}[ht]
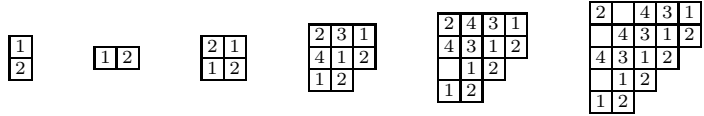

\begin{center}
\begin{ytableau}
1 \\
2
\end{ytableau}
\qquad
\begin{ytableau}
1 & 2
\end{ytableau}
\qquad
\begin{ytableau}
2 & 1 \\
1 & 2
\end{ytableau}
\qquad
\begin{ytableau}
2 & 3 & 1 \\
4 & 1 & 2 \\
1 & 2
\end{ytableau}
\qquad
\begin{ytableau}
2 & 4 & 3 & 1 \\
4 & 3 & 1 & 2 \\
\ & 1 & 2 \\
1 & 2
\end{ytableau}
\qquad
\begin{ytableau}
2 & \ & 4 & 3 & 1 \\
\ & 4 & 3 & 1 & 2 \\
4 & 3 & 1 & 2 \\
\ & 1 & 2 \\
1 & 2
\end{ytableau}
\end{center}
\caption{Case $(\bar\delta_2, \bar\delta_3, \bar\delta_4) = (0,2,2)$:
Building Blocks}
\label{fig:022building}
\end{figure}

In Phase~2, if the building blocks are sufficiently large,
then when we concatenate
them, we simply add two more 3's along antidiagonal $\delta_1-2$
and two more 4's along antidiagonal $\delta_1-3$.
The smaller building blocks are handled as shown in
Figure~\ref{fig:022concat}.
\begin{figure}[ht]
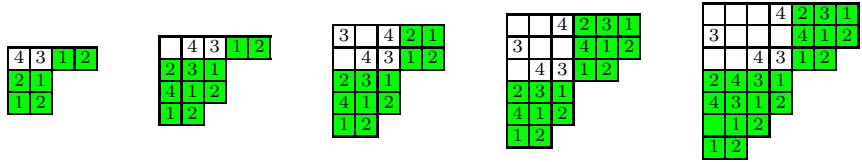

\begin{center}
\begin{ytableau}
4 & 3 & *(green) 1 & *(green) 2 \\
*(green) 2 & *(green) 1 \\
*(green) 1 & *(green) 2
\end{ytableau}
\qquad
\begin{ytableau}
\ & 4 & 3 & *(green) 1 & *(green) 2 \\
*(green) 2 & *(green) 3 & *(green) 1 \\
*(green) 4 & *(green) 1 & *(green) 2 \\
*(green) 1 & *(green) 2
\end{ytableau}
\qquad
\begin{ytableau}
3 & \ & 4 & *(green) 2 & *(green) 1 \\
\ & 4 & 3 & *(green) 1 & *(green) 2 \\
*(green) 2 & *(green) 3 & *(green) 1 \\
*(green) 4 & *(green) 1 & *(green) 2 \\
*(green) 1 & *(green) 2
\end{ytableau}
\qquad
\begin{ytableau}
\ & \ & 4 & *(green) 2 & *(green) 3 & *(green) 1\\
3 & \ & \ & *(green) 4 & *(green) 1 & *(green) 2 \\
\ & 4 & 3 & *(green) 1 & *(green) 2 \\
*(green) 2 & *(green) 3 & *(green) 1 \\
*(green) 4 & *(green) 1 & *(green) 2 \\
*(green) 1 & *(green) 2
\end{ytableau}
\qquad
\begin{ytableau}
\ & \ & \ & 4 & *(green) 2 & *(green) 3 & *(green) 1\\
3 & \ & \ & \ & *(green) 4 & *(green) 1 & *(green) 2 \\
\ & \ & 4 & 3 & *(green) 1 & *(green) 2 \\
*(green) 2 & *(green) 4 & *(green) 3 & *(green) 1 \\
*(green) 4 & *(green) 3 & *(green) 1 & *(green) 2 \\
*(green) \ & *(green) 1 & *(green) 2 \\
*(green) 1 & *(green) 2
\end{ytableau}
\end{center}
\caption{Case $(\bar\delta_2, \bar\delta_3, \bar\delta_4) = (0,2,2)$:
Phase~2}
\label{fig:022concat}
\end{figure}
% This completes the proof of Theorem~\ref{thm:L4}.
\end{proof}

\section{Concluding Remarks}

We believe that there is some hope that the coloring algorithms
in the proof of Theorem~\ref{thm:L4} can be generalized to prove
the full conjecture, though at the moment, there remain several
annoying ad hoc steps.

Another possible approach to proving the conjecture,
which we have not investigated, is to examine closely what
happens if one takes~$B_\lambda$ and applies the known
general algorithms for constructing maximum-cardinality
unions of $i$ matchings in a bipartite graph,
for $i=1,2,3\ldots\,$. For an arbitrary bipartite graph,
these unions of $i$ matchings will not ``nest'' nicely,
but maybe one can see how to make them nest in the case of~$B_\lambda$.

It is also natural to ask more generally which classes of graphs
are CDS-colorable. We suspect that CDS-colorability is NP-complete,
but perhaps there are interesting classes of CDS-colorable graphs.
Griggs~\cite{Gri} showed that incomparability graphs of
symmetric chain orders are CDS-colorable.
It is an unpublished result of the first author
(proved independently by others~\cite{F-S,MMS})
that indifference graphs (also known as unit interval graphs)
are CDS-colorable, and that the coloring can be constructed
by scanning the vertices from left to right,
assigning the lowest available color at each step.
It may also be possible to extract some examples from the
literature on integer multiflows~\cite[Volume~C]{S};
there is no general ``max-flow min-cut'' theorem in this
setting, but some partial results are known, which may be relevant.

Finally, for those with an interest in Rota's basis conjecture,
here is an ambitious generalization to matroids of the Latin Tableau Conjecture.
Let $\lambda$ be a partition, let $\delta$ denote the CDS of~$G_\lambda$,
and let $\mu=\delta'$ be the conjugate partition (i.e.,
the diagram of~$\mu$ is obtained from the diagram of~$\delta$
by transposing rows and columns).
Let $M$ be any matroid that is a disjoint union of independent sets
of cardinalities $\mu_1, \mu_2, \ldots\,$.
Then we may conjecture that there is a way to place the elements of~$M$
into the diagram of~$\lambda$, one element per box,
such that the rows and the columns are all independent sets of~$M$.
We have not made any serious attempt to disprove this conjecture,
so maybe it is false, but even if it is false, perhaps the counterexamples
will help illuminate the Latin Tableau Conjecture
and/or Rota's basis conjecture.

\section{Acknowledgments}

We thank Katharine C. Walker for helpful conversations and ideas in
the early stages of this research project.

\end{document}